\pdfoutput=1
\documentclass[%
 reprint,
superscriptaddress,
 amsmath,amssymb,
 aps,
 mathtools,amsthm,
]{revtex4-2}

\usepackage{graphicx}
\usepackage{dcolumn}
\usepackage{bm}
\usepackage{hyperref}
\usepackage{xcolor}
\usepackage{amsmath}
\usepackage{mathtools}
\usepackage{dsfont}
\usepackage{braket}
\usepackage{multirow}
\usepackage{tabularx}
\usepackage{bbm}
\usepackage{grffile}

\DeclarePairedDelimiterX{\norm}[1]{\lVert}{\rVert}{#1}

\newcommand{\BR}{\mathbb{R}}

\newtheorem{theorem}{Theorem}

\usepackage[ruled,vlined,noline,linesnumbered]{algorithm2e}

\newtheorem{definition}[theorem]{Definition}
\newtheorem{example}[theorem]{Example}

\newtheorem{notation}{Notation}[section]

\newtheorem{proposition}[theorem]{Proposition}
\newtheorem{remark}[theorem]{Remark}

\newenvironment{proof}[1][Proof]{\textbf{#1.} }{\ \rule{0.5em}{0.5em} \vspace{1ex}}

\newcommand{\R}{\mathbb{R}}

\newcommand{\bx}{x}

\newcommand{\bs}{\boldsymbol}
\usepackage{dsfont}

\begin{document}

\preprint{APS/123-QED}

\title{Polynomial Matrix Inequalities within Tame Geometry}

\author{Christos Aravanis}
\email{c.aravanis@sheﬃeld.ac.uk}
\affiliation{The University of Sheffield International College, United Kingdom}

\author{Johannes Aspman}
\email{aspmanj@maths.tcd.ie}
\affiliation{School of Mathematics and Hamilton Mathematics Institute, \\Trinity College Dublin, Ireland}

\author{Georgios Korpas}
\email{georgios.korpas@fel.cvut.cz}
\author{Jakub Marecek}%
 \email{jakub.marecek@fel.cvut.cz}
\affiliation{%
Department of Computer Science and Artificial Intelligence Center, \\ Czech Technical University in Prague, Czech Republic 
}%


\date{\today}

\begin{abstract}
Polynomial matrix inequalities can be solved using hierarchies of convex relaxations,
pioneered by Henrion and Lassere. 
In some cases, this might not be practical, and one may need to resort to methods 
with local convergence guarantees, whose development has been rather \emph{ad hoc}, so far.
In this paper, we explore several alternative approaches to the problem, 
with non-trivial guarantees available using results from tame geometry. 
\end{abstract}

\maketitle


\section{Introduction}\label{sec:level1}

Polynomial matrix inequalities (PMI) are a class of (generally) non-convex optimization problems  \cite{VanAntwerp2000,Henrion2006},
with 
extensive applications in control theory \cite{henrion2005solving,Henrion2006},
structural engineering and design optimization \cite{tyburec2021global},
and 
quantum information theory \cite{bondar2022recovering,Aravanis2022}.
Unfortunately, even the verification of local optimality of PMI is an NP-hard problem \cite{Murty1987},
and polynomial-time solvers with guarantees of global convergence to global optima are hence unlikely. 

Nevertheless, there are several approaches to solving PMIs in the literature, which we survey below.
These include polynomial-time methods \cite{jarre2012elementary,kocvara2005penbmi,henrion2005solving,diehl2006loss,freund2007nonlinear,sun2008rate,andreani2020optimality,andreani2021optimality}, which have only guarantees of local convergence to local optima,
developed in an ad-hoc fashion. 
There are also methods based on solving ever larger convexifications \cite{Scherer2005,Henrion2006,D2018ANO,Dinh2021,Zheng_2021}, which have guarantees of asymptotic convergence to global optima.
A number of branch-and-bound  \cite{goh1995global,apkarian2000robust} and branch-and-cut \cite{fukuda2001branch} approaches have also been considered. 
Overall, however, progress has been rather intermittent.

In this paper, we explore several approaches to the problem, which have not been explored previously, 
with guarantees on their performance (e.g., local convergence, rates thereof) available using results from tame geometry \cite{davis2020stochastic,bolte2021conservative}, 
a fast-developing field on the interface of topology, functional analysis, and optimization.
We also try to reason about the suitability of the proposed approaches in various settings. 

Technically, this relies on an extension of the reasoning about the
tame nature of the feasible set of a constrained optimization problem 
and the graph of an unconstrained optimization problem 
to reasoning about tame representation of a (generalized) Lagrangian of a constrained optimization problem,
which could be of some interest beyond PMI.

\subsection{The Role of Tame Geometry}

In \cite{grothendieck1997around}, Grothendieck advocated for the construction of a theory of tame topological and geometrical structures motivated by several issues within analysis. He has also initiated a program for its development. 
An important notion for the development of tame geometry is the stratification of smooth manifolds, that is, colloquially, the breaking up of manifolds into submanifolds of lower dimension. As a matter of fact, many ``nice" topological spaces, including moduli spaces of curves and, more prominently, solution sets of systems of polynomial equations and semialgebraic sets, provide a landscape where stratifications are straightforward and understandable. Semialgebraic sets are very natural in convex optimization theory and are prototypical examples of sets with tame topology, also known as tame sets. Tame sets are now understood more broadly in the context of o-minimal structures, which we introduce in Sec. \ref{sec:tame_geometry}. 
As we will explain, a set is called tame if it is definable on an o-minimal structure, and similarly, a function is tame if its graph corresponds to a tame set. Many different o-minimal structures have been constructed that generalize semialgebraic sets.
Because of the tameness of these spaces, important results in analysis and optimization theory follow. 

Tame geometry and the theory of o-minimal structures have been used extensively for obtaining new results in geometric topology \cite{vandenDries1996}, analysis and subgradient calculus \cite{bolte2007clarke,Attouch2011,bolte2020long,bolte2021conservative} as well as complex algebraic geometry and Hodge theory \cite{Bakker2018,Bakker2018b,Bakker2020,Bakker2021} with implications reaching as far out as consistency arguments for effective field theories in the context of string theory \cite{Grimm2021,Grimm2022}.

Tame geometry has also found important applications in optimization theory. It is now understood that semialgebraic functions, semianalytic functions, and more generally functions definable in o-minimal structures form a class of functions in which first-order optimization methods are applicable \cite{davis2020stochastic,bolte2021conservative} in the following sense: 
for any locally-Lipschitz Whitney stratifiable function, there exists a descent guarantee along any subgradient trajectory, which, in turn, results in convergence guarantees for stochastic subgradient methods \cite{Helton2003}. 
Similarly, Newton's method converges fast \cite{bolte2009tame} on locally-Lipschitz tame functions, which turn out to be semismooth.
To appreciate this, consider the three qualitatively different types of behavior, as far as convergence guarantees are concerned:
\begin{itemize}
    \item Convex and tame: convergence to optimality along a finite trajectory \cite{Kurdyka2000}.
    \item Convex and non-tame: many methods fail to converge \cite{Bolte2020}, due to the lack of a finite trajectory.
    \item Tame and non-convex: methods adding noise to subgradients can produce global optima asymptotically. Without the addition of noise,  we obtain convergence to a critical point in finite time. 
\end{itemize}
See \cite{davis2020stochastic} and \cite{bolte2021conservative} for two state-of-the-art analyses. We expect many more algorithms and their analyses to appear,
especially in relation to optimization in deep learning, where an important class of problems \cite{bolte2021conservative} has
been shown to be tame. 


In this paper, we introduce tame representations of non-convex polynomial matrix inequality problems by formulating the generalized Lagrangian $\widetilde{\mathcal{L}}$ associated to the Lagrangian $\mathcal{L}$ of a PMI, as exposed in Eqs. \eqref{eq:LagrangianA} and \eqref{eq:LagrangianB}. Our approach is inspired by the representability of subsets of $\mathbb{R}^n$ as linear matrix inequalities and semidefinite programs, except reversed in direction. 

To obtain some intuition, consider that representability in linear matrix inequalities \cite{Helton2003}, sometimes also known as positive-definite-representable functions \cite[Section 6.4.2]{nesterov1994interior}.
This amounts to finding symmetric matrices $\{A_i\}_{i=0}^n \in \mathbb{R}^{m\times m}$ such that the set defined as $S = \{ x\in \mathbb{R}^n \,|\, A\succeq 0 \}$ exists, with $A = A_0+\sum_{i}A_ix_i$. 
Notice that one could easily extend these results towards representability in polynomial matrix inequalities, when considering the fact that polynomial matrix inequalities are representable in linear matrix inequalities.
This comes from the work of Henrion and Lassere \cite{Henrion2006}, which we explain in Section \ref{sec:convexifications}.

The reverse problem can also be formulated, but it only makes sense if one asks for some further requirements. Given $A \succeq 0$, it asks whether there exists a set $S$, defined as above, such that it has a certain topological or geometrical property. We are interested in such a question where this special property is the tameness that we discussed in the previous paragraph. We define the notion of \emph{tame representation of a polynomial matrix inequality} which, in general, corresponds to finding a reformulation of the polynomial matrix inequality such that the graph of the, unconstrained, objective function is a tame set. We provide three examples of such representations using the characteristic polynomial of the polynomial matrix inequality, using factorization for a bilinear matrix inequality, and the log-det barrier. 

The interest and potential in tame representations originate in the ability to use first-order methods \cite[e.g.]{apkarian2000robust,kovcvara2003pennon,kocvara2005penbmi}, with theoretical convergence guarantees. For example, locally Lipschitz functions defined over an o-minimal structure are semismooth, and within this class the asymptotic behavior of the $k$-th step of the Newton method is known \cite{bolte2010characterizations}. More recently, fundamental work by Bolte and Pauwels \citep{bolte2021conservative} has shown that first-order methods, such as the mini-batch stochastic gradient descent, can be used with theoretical convergence guarantees for tame functions.

This paper is structured as follows. In Sec. \ref{sec:tame_geometry} we provide a self-contained background to o-minimal structures with various examples aiming to provide intuition to the reader unfamiliar with these notions. In Sec. \ref{sec:PMIs} we provide the necessary background on polynomial matrix inequalities and review the current state-of-the-art family of approaches for solving them. Sec. \ref{sec:tamereps} contains the main contributions of this paper with several definitions, propositions, and theorems with respect to the tame representations of polynomial matrix inequalities. Finally, we summarize and conclude this paper in Sec. \eqref{sec:discussion}.

\section{Background: Semialgebraic and Tame Structures}\label{sec:tame_geometry}
The goal of this paper is to bring together tame structures and PMIs. In this section, we collect the necessary definitions about tame structures. Since semialgebraic structures are core examples of tame structures, we include related definitions too.

\begin{notation}
In the following, the $n$-dimensional Euclidean space is denoted by $\R^{n}$. For a subset $X \subseteq \R^{n}$ denote by $\mathcal{P}(X)$ its powerset, that is, the set of all subsets.  The graph of a map $f\colon X\to Y$ is the set $\Gamma(f)$ defined by \[\Gamma(f)=\Big\{(x,y)\in X\times Y:y=f(x)\Big\}.\]

\end{notation}

\subsection{Semialgebraic structures }

In this section, we introduce the notions of semialgebraic sets and of semialgebraic functions necessary for what follows. For more details, see \cite[Chapter 2]{bochnak2013real}.

A subset $X\subseteq \R^{n}$ is said to be semialgebraic if it is a Boolean combination of sets of solutions of polynomial equations and polynomial inequalities with real coefficients. More precisely, if $X$   can be represented as
\begin{align}
    X = \bigcup_{i=1}^s \bigcap_{j=1}^t \left\lbrace \bx\in \R^p: P_{ij}(\bx) \star 0 \right\rbrace,
\end{align}
where $P_{ij}$ are real polynomials, $\star$ is one of $>, <, =$ and $s$ and $t$ are finite numbers.

Examples of semialgebraic sets are algebraic sets, meaning sets of all zeros of real polynomials on $n$ variables. More examples of semialgebraic sets are obtained by taking finite unions, finite intersections, Cartesian products, and complements of semialgebraic sets. 
 
A function $f\colon X\to Y$ between semialgebraic sets is semialgebraic if its graph $\Gamma(f)$ is a semialgebraic set. For example, real polynomials in $n$ variables are semi-algebraic functions. It is worth mentioning that, naturally, the exponential map, $\exp:\R\to\R$, is not a semialgebraic function.








\subsection{Tame structures and definable maps}

In this section, we introduce tame structures also known as $o$-minimal structures which will be of particular interest in Section \ref{sec:tamereps}.  We start with the definition of an $\R$-structure and then a tame structure is defined as an  $\R$-structure which satisfies an extra condition. For more details on tame structures and its applications, see \cite{Dries98}, \cite{vandenDries1996} \citep[Chapter 1]{Nicol}. 

\begin{definition} \label{def:R-str}
An \textbf{$\R$-structure} $S$ is a collection of sets $\{S^{p}\}_{p\ge 1}$ where each $S^{p}$ is a subset of the powerset of $\R^{p}$ such that the following hold:

\begin{itemize}
    \item [$A_{1}:$] Each $S^{p}$ contains all zero sets of finite collections of polynomials in real variables $p$.
    
     \item [$A_{2}:$] For every linear map $L\colon \R^{p} \to \R$, the half-plane $\{x \in \R^{p}| L(X)\ge 0\}$ belongs to the set $S^{p}$.
     
     \item [$P_{1}:$] For every $p\ge 1$, the $S^{p}$ is closed under set-theoretic unions, intersections, and complements.
     
      \item [$P_{2}:$] If $X \in S^{p}$ and $Y \in S^{q}$ then $X\times Y\in S^{p+q}$.
      
      \item [$P_{3}:$] If $T\colon \R^{p}\to \R^{q}$ is an affine map and $X \in S^{p}$, then $T(X) \in S^{q}$.
     
\end{itemize}

\end{definition}



\begin{example}\label{ex:semialg}
One of the most fundamental examples of an $\R$-structure is the collection of real semialgebraic sets $\mathbb{R}_{\rm{semialg}}$, given as polynomial inequalities in real variables. 
\end{example}

\begin{definition}
Let $S$ be a fixed $\R$-structure.

\begin{enumerate}
    \item A set $X$ is said to be \textbf{$S$-definable} if $X$ belongs to some $S^{p}$. 
    
    \item A map $f\colon X \to Y$ between $S$-definable sets is said to be an \textbf{$S$-definable map} if its graph $\Gamma(f)$ is $S$-definable. 
\label{def:definable}
\end{enumerate}
\end{definition}

\begin{example} In the following several examples of $S$-definable sets and $S$-definable functions are provided.

\begin{enumerate}
    \item Any polynomial with real coefficients is a $S$-definable map.
    
    \item Let $X\xrightarrow{f}Y\xrightarrow{g}Z$ be two $S$-definable maps. Then its composition $g\circ f$ is an $S$-definable map.
    
    \item For a $S$-definable map $f(x, y)$, the  $ \inf_y f(x, y)$ is $S$-definable.

    \item The image and pre-image of a definable set under a definable map are both definable sets.
    
    \item \label{inverse_definable} The inverse of an injective and definable function is also a definable function, see \cite[Lemma 2.3 (iii)]{Dries98}
    
\end{enumerate}

\end{example}

New $\R$-structures are obtained from old ones as follows: for an $\R$-structure $S$ and a collection $A=\{A_{p}\}_{p\ge 1}$  where $A_{p} \in \mathcal{P}(\R^{p})$, denote by $S(A)$ the smallest structure that contains $S$ and the sets $A_{p}$. Then, we say that $S(A)$ is obtained from $S$ by adjoining the collection $A$. 

\begin{example}\label{ex: res Analytic} A \textbf{restricted real analytic function} is
a function $f\colon \R^{p} \to \R$ with the property that there exists a real analytic function $f$ defined in an open neighborhood $U$ of the hypercube $C_n \coloneqq [-1,1]^n$ such that
\begin{align}
   f &=  \begin{cases}
      \widetilde{f}, &   \text{if $x \in C_{n}$} \\
       0, &   \text{if $x \in \R^{p} \backslash  C_{n}$}.
    \end{cases}
\end{align}

Adjoining the structure $\mathbb{R}_{\rm{semialg}}$ with the graphs of all restricted real analytic functions, one obtains the $\R$-structure $\R_{\rm{an}}$.
\end{example}
 
\begin{example} \label{ex: exp}
Adjoining the structure $\mathbb{R}_{\rm{semialg}}$ with the graph of the exponential map $\R \to \R, t\mapsto e^{t}$ one obtains the $\R$-structure $\mathbb{R}_{\rm{exp}}$.
\end{example}

We are really interested in structures that are tame, also known as $o$-minimal which we will define now.

\begin{definition} 
\label{def:tamestructure}
Let $S$ be an $\R$-structure, as given in Definition \ref{def:R-str}. Then, $S$ is said to be a \textbf{tame structure} or an \textbf{o-minimal structure} if the following property is satisfied:

\begin{itemize}
    \item [$T:$] Any set $X \in S^{1}$ is a finite union of singletons $\{*\}$ and open intervals $(a,b)$ where $-\infty \le a < b \le \infty$.
\end{itemize}

\end{definition}

In other words, tame structures are $\R$-structures which satisfy an extra property $T$ as defined above.

\begin{example}Examples of tame structures are the following.
\label{ex:tamestructure}
\begin{enumerate}
    \item The collection $\mathbb{R}_{\rm{semialg}}$ of real semialgebraic sets, defined in  Example \ref{ex:semialg}.
    \item  The $\R$-structure $\mathbb{R}_{\rm{an}}$, defined in Example \ref{ex: res Analytic}.
    \item The $\R$-structure  $\mathbb{R}_{\rm{exp}}$, defined in Example \ref{ex: exp}.
\end{enumerate}

\end{example}

\begin{notation}
A function definable in some tame structure $S$ will be referred to as a tame map.
\end{notation}

\begin{figure}
    \centering
    \includegraphics[scale=0.82]{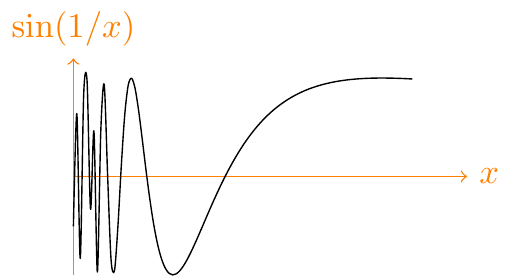}
    \includegraphics[scale=0.82]{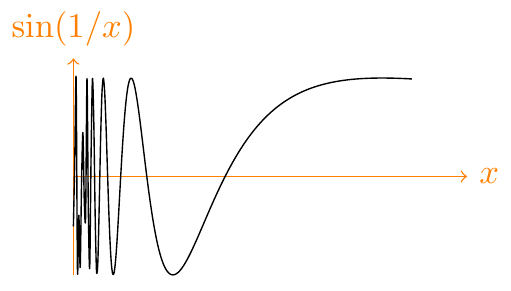}
    \includegraphics[scale=0.82]{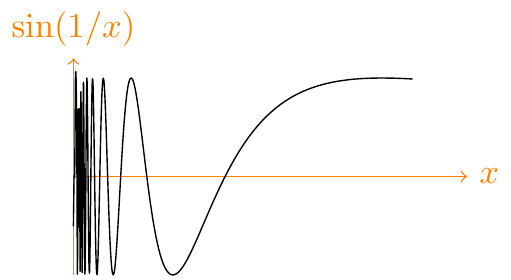}
    \includegraphics[scale=0.82]{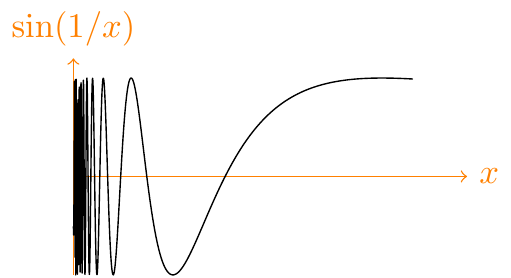}
    \caption{Consider the function $f(x)=\sin(1/x)$ and its graph $\Gamma(f)=\left\{\left(x, \sin \frac{1}{x}\right): x \in(0,1]\right\} \cup\{(0,0)\}$.  $\Gamma(f)$ is not a tame structure in $(0,1)$ since its closure $\bar{\Gamma}(f) = \partial \Gamma(f) \cup \Gamma(f)$ is connected but not path-connected, the boundary $\partial\Gamma(f)$ has the same dimension as $\Gamma(f)$ and it is not a stratifiable manifold.  In this figure we sampled $f$ 100, 250, 500 and 1000 times (left,right,top,bottom). It is clear that near the $y$-axis, this graph is very badly behaved. This example is borrowed from \cite{Klinger}. A similar example is that of the Volterra function \cite{PonceCampuzano2015}.}
    \label{fig:sine}
\end{figure}
\begin{figure}
\centering
\includegraphics[scale=1.0]{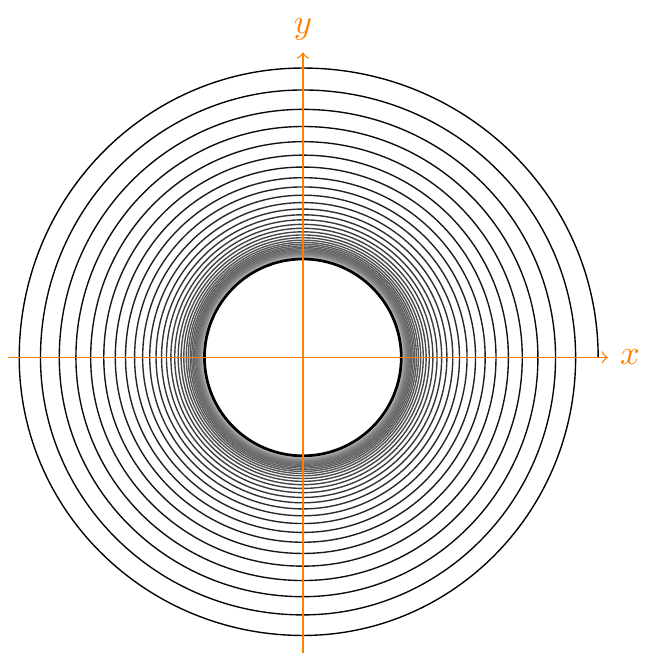}
\caption{Graph $\Gamma(\gamma)$ of the parametrized curve $\gamma(s)= \braket{(1+s^{-1})\cos(s),(1+s^{-1})\sin(s)}\cap \mathbb{S}^1$. This is a similar example to Fig. \ref{fig:sine} in the sense that it is a non-path connected space (there is no path connecting the path connected spiral to the unit circle) and thus not definable in an o-minimal structure.}
   \label{fig:my_label}
\end{figure}

\begin{figure}
    \centering
    \includegraphics{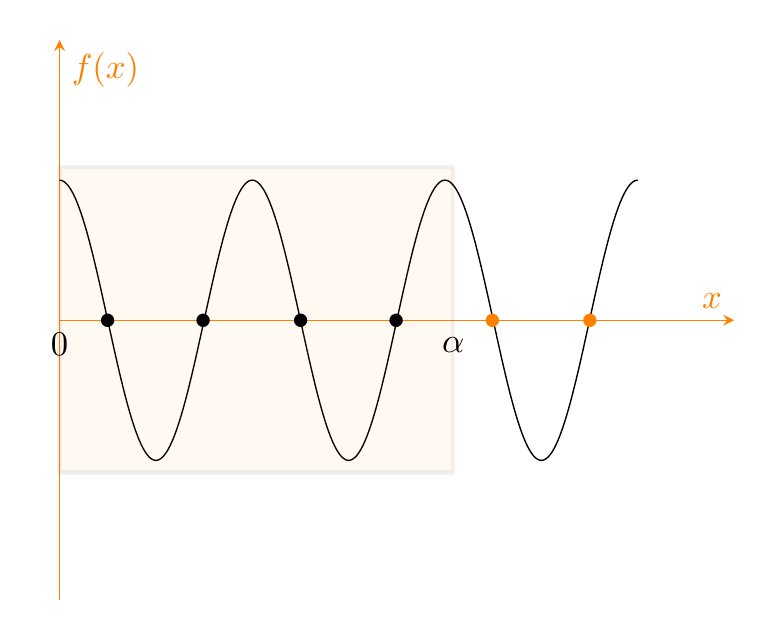}
    \caption{The graph of $f(x) =\sin(x)$ is definable for $x \in [0,\alpha]$, $\alpha \in \mathbb{R}$ with finite number of points (the zeros of $f(x)$). Recall that definability requires that any projection on $\mathbb{R}$ maps the graph of the function to finitely many points which will not be the case for $\alpha \to \infty$.}
    \label{fig:sin}
\end{figure}






\section{Background: polynomial matrix inequalities}\label{sec:PMIs}
This section is devoted to introducing the types of problems we are interested in, namely polynomial matrix inequalities. After introducing the general definitions, we discuss some related previous results. 

\subsection{Generalities}


Let $\mathcal{S}^{m} \subset \mathbb{R}^{m \times m}$ denote the space of real $m \times m$ symmetric matrices.  We denote the subset of positive definite (semidefinite) $m \times m$ matrices by $\mathcal{S}^{m}_{++} \subset \mathcal{S}^{m}$ ($\mathcal{S}^{m}_{+} \subset \mathcal{S}^{m}$). For a matrix $A \in \mathcal{S}^m_{++}$ ($A\in \mathcal{S}^m_{+}$) we write $A\succ 0$ ($A\succeq 0$). Finally, a semialgebraic set $\mathcal{K}$ is defined as
\begin{equation}
\label{e:semialgebgraic-set}
\mathcal{K} = \left\{ x \in \mathbb{R}^k:\; g_1(x)\geq 0,\, \ldots,\, g_q(x) \geq 0 \right\},
\end{equation}
where $g_1,\ldots, g_q \in \mathbb{R}[x]$.

\begin{definition}[PMI] Let $y\in \mathbb{R}^\ell$, $x\in \mathbb{R}^k$,  a polynomial $b \in \mathbb{R}[y]$ and a polynomial matrix  $P: \mathbb{R}^k \times \mathbb{R}^\ell \to \mathcal{S}^m$. 
Then, a polynomial matrix inequality optimization problem is the following optimization problem:
\begin{align}\label{problem:pmi}
    \begin{array}{rll}
B^{*} \coloneqq & \inf_{y\in \mathbb{R}^l}  b(y) \quad  \text { such that } \quad   P(x,y) \succeq 0.
\end{array}
\end{align}
\end{definition}

We will sometimes refer to the polynomial $b$ in the above definition as a \emph{cost} function and to a polynomial matrix inequality optimization problem as a PMI problem.



In general,  $P(x,y)$ is defined by $P_0(x) - \sum_{i=1}^\ell P_i(x)y_i$ and $\| y\|_{\infty}<1$. In general, verifying polynomial matrix inequalities is an NP-hard problem \cite{Murty1987}, and Prob. \eqref{problem:pmi} is generally intractable.

\begin{definition}
 A feasible set for a PMI is the subset of the space of search variables that satisfies the PMI condition. In the notation of \eqref{problem:pmi}, this means the subset of $\BR^l$ such that the optimization problem has a solution that satisfies the PMI constraint $P(x,y)\succeq 0$.
\end{definition}

A few remarks are due for some special cases of Prob. \eqref{problem:pmi}.

\begin{enumerate}
    \item[\emph{(1)}] If $\deg(b) = \deg(P) = 1$, then Prob. \eqref{problem:pmi} corresponds to a convex linear matrix inequality (LMI). Precisely, given $y\in \mathbb{R}^n$ and $\{ P_i\}_{i=0}^n \in \mathcal{S}_+^m$, an LMI takes the form
    \begin{align}\label{eq:LMIex}
        P_0 + y_1P_1 + \ldots + y_n P_n \succeq 0,
    \end{align}
    and the set of real vectors $y \in \mathbb{R}^n$ that satisfy the matrix inequality above define a spectrahedron \cite{Henrion2015}.

  The set
of real vectors x such that the pencil is positive semidefinite is a convex semialgebraic set called spectrahedron, described by a linear matrix inequality (LMI)

    \item[\emph{(2)}] If $\deg(P)=2$ and no square terms appear in it, then $P$ takes the form \begin{align}\label{}
    P_0 + \sum_{i=1}^{m} x_i F_i + \sum_{j=1}^{n} y_j G_j + \sum_{i=1}^{m} \sum_{j=1}^{n} x_i y_j H_{ij}
\end{align}for $F_i,G_j,H_{ij}\in \mathcal{S}^m$. In this case, Problem \eqref{problem:pmi} is referred to as a bilinear matrix inequality (BMI). 
    
\item[\emph{(3)}] If the matrix $P(x,y)=:X$ is a symmetric positive definite matrix,  i.e. $X \in \mathcal{S}_{++}^m$, then Problem \eqref{problem:pmi} is transformed to
\begin{align}\label{problem:sdmi}
    \begin{array}{rll}
X^{*} \coloneqq & \inf_{X} & Q(X) \quad  \text { such that } \quad  X \succ 0, \
\end{array}
\end{align}
where $Q$ is a matrix-valued polynomial. Typically, this type of problem comes with some additional constraint, for example
\begin{equation}
\begin{aligned}\label{problem:sdmiTr}
X^{*} \coloneqq & \inf_{X}  Q(X) \\  
\text {s.t. }  &X \succeq  0,\\ 
&\text{Tr}\left( X\right)= 1.
\end{aligned}
\end{equation}

\item[\emph{(4)}]
If $P$ has a chordal sparsity graph with maximal cliques $\mathcal{C}_1,\ldots, \mathcal{C}_t$, then $P$ can be decomposed as \cite{Zheng_2021}
\begin{align} \label{problem:sparse}
    P(x,y) = S_0(x) + \sum_{j=1}^k g_j(x,y)S_j(x),
\end{align}
for $m\times m$ sum-of-squares (SOS) polynomial matrices $\{ S_i\}_{i=0}^k$ where $S_i(x) = H(x)^\top H(x)$ for some polynomial matrix $H(x)$.
\end{enumerate}

\subsection{Convex Relaxations of PMIs}
\label{sec:convexifications}

The objective function of a PMI is a polynomial, by definition, on an indeterminate $X\in \mathcal{S}_+^m$. One can leverage the power of polynomial optimization techniques, whose overview is in Table~\ref{table:summary-pop}, to solve PMIs. Specifically, one can use convergent relaxations of PMIs \cite{Henrion2006} based on the scalar convergent relaxations of SDPs \cite{lasserre2001global}. This method is global because it can solve PMI problems even when
finite convergence occurs, while, in parallel,  it provides a numerical certificate of global optimality. However, it is also local, given that the convex relaxations utilize SDPs. 


In this section, we will review the moment relaxations of Lassere \cite{lasserre2001global} for PMI problems such as Problem \eqref{problem:pmi}, following
Henrion and Lassere \cite{Henrion2006}. 
Let $\bs b_d$ denote the vector of all possible monomials, up to degree $d$, constructed out of the entries of $x \in \mathbb{R}^n$. That is, the entries of $\bs b_d$ are monomials in $\mathbb{R}[x_1,\ldots, x_n]$, for $x_i \in \mathbb{R}$,
\begin{equation}
    \begin{aligned}
    \bs{b}_d(x)^\top &= (1~x_1~x_2~\ldots~x_n~x_1^2~x_1x_2~x_1x_n~\ldots \\
             & \qquad x_2~x_2x_3~\ldots~{x_n^2} ~\ldots~\ldots x_n^d)^\top.
\end{aligned}
\end{equation}
This vector can be used to construct the order-$d$ generalized Hankel matrix
\begin{align}\label{HankelX}
    \bs{b}_d(x) \bs{b}_d(x)^\top=\left(\begin{array}{ccccc}
1 & x_{1} & x_{2} & \cdots & x_{n}^{d} \\
x_{1} & x_{1}^{2} & x_{1} x_{2} & \cdots & x_{1} x_{n}^{d} \\
x_{2} & x_{1} x_{2} & x_{2}^{2} & \cdots & x_{2} x_{n}^{d} \\
\vdots & \vdots & \vdots & \ddots & \vdots \\
x_{n}^{d} & x_{1} x_{n}^{d} & x_{2} x_{n}^{d} & \cdots & x_{n}^{2 d}
\end{array}\right). 
\end{align}
For what follows, consider the limit $d\to \infty$ such that $\lim_{d\to \infty}\bs{b}_d \coloneqq \bs{b}$. Let $y = \{y_a\}_{a\in \mathbb{N}^n}$ be a sequence of indeterminates labeled with the entries of $\bs{b}(x)$. The sequence $\{y_a\}_{a\in \mathbb{N}^n}$ is referred to as the \emph{moment variable}, by defining $y_a \coloneqq \int {x^a d\mu}$, for a Borel measure $\mu$.

Similarly to the order-$d$ generalized Hankel matrix, we can define the order $d$-moment matrix
\begin{align}\label{eq:gen_Hank}
    M_{d}(y):=\left(\begin{array}{cccc}
y_{(0, \ldots, 0)} & y_{(1, \ldots, 0)} &  \cdots & y_{(0, \ldots, 1)} \\
y_{(1, \ldots, 0)} & y_{(2, \ldots, 0)} &  \cdots & y_{(1, \ldots, d)} \\
y_{(0,1, \ldots, 0)} & y_{(1,1, \ldots, 0)} & \cdots & y_{(0,1, \ldots, d)} \\
\vdots & \vdots &  \ddots & \vdots \\
y_{(0, \ldots, d)} & y_{(1, \ldots, d)} &  \cdots & y_{(0, \ldots, 2 d)}
\end{array}\right). 
\end{align}
As before, for what follows, consider the limit $d\to \infty$ and denote $\lim_{d\to \infty}M_d(y) \coloneqq M(y)$. Similarly, consider an infinite-dimensional basis represented by $\bs{b}(x)$ or, equivalently, $y_{a}$.
Any polynomial $p\in \mathbb{R}[x_1,\ldots, x_n]$ can then be identified with the vector of its coefficients labelled by $a\in \mathbb{N}^n$ as $p(x) = \sum_{a\in \mathbb{N}^n} p_a x^a = \braket{\bs{p},\bs{b}}$. Similarly to this construction, we can define the linear mapping $L_{y}(p) = \braket{p,y}$, while for polynomials $p,q\in \mathbb{R}[x_1,\ldots, x_n]$ we can define the bilinear mapping $L_{y}(p,q) = \braket{p,M(y)q}$. It is straightforward to confirm that Eq. \eqref{eq:gen_Hank} is equivalent to 
\begin{align}
    M(y) = \int \bs{b}\bs{b}^\top d\mu
\end{align}
It immediately follows that, for any two polynomials $p$ and $q$, $L_y(pq) = \int pq d\mu$.
For $d$ finite, corresponding to monomial basis $\bs{b}_d \subset \bs{b}$ of degree at most $d$ and moment matrix $M_d \subset M$ of order $d$, it follows that $M_d(y) \succeq 0$ since for any $p \in \mathbb{R}_d[x]$, $\braket{p,p}_y = \braket{\bs{p}, M_d(y) \bs{p}} = \int p^2 d\mu \geq 0$. Similarly, for more general moments, if the measure $\mu$ of $y$ is in the semialgebraic set defined as $\{ x\in \mathbb{R}^n \, | \, g_i(x) \geq 0\}$, $\braket{p,p}_{g_iy} = \braket{\bs{p}, M_d(g_iy) \bs{p}} = \int g_ip^2 d\mu \geq 0$, for all $i$, and as a result we have that all truncated moment matrices are positive semidefinite.

For matrix variables, we consider the PMI optimization Prob. \eqref{problem:pmi}, where $P: \mathbb{R}^{n} \rightarrow \mathcal{S}^{m}$ such that each entry $P_{i j}(x)=P_{j i}(x)$ of the matrix $P(x)$ is a polynomial in $\mathbb{R}[x]$. The semialgebraic set $\mathcal{K}$ of Eq. \eqref{e:semialgebgraic-set} can be equivalently written as
\begin{align}
    \mathcal{K}:=\left\{x \in \mathbb{R}^{n}: G(x) \succeq 0\right\} .
\end{align}
For a polynomial mapping $P: \mathbb{R}^{n} \rightarrow \mathbb{R}^{m}$ of degree at most $k$, we write $x \mapsto P(x)=\bs{P} \bs{b}_k(x) \in \mathbb{R}^{m}$, for some $m \times s_{k}$ matrix $\mathbf{P}$, where $s_{k}$ is the dimension of the vector space $\mathbb{R}_{k}[x]$. 

Similarly to the definition of the moment matrix for scalar variables, let $M_{k}(y) \coloneqq \{y_{\alpha+\beta} \}_{|\alpha|,|\beta|\leq k}$, be the order $k$ moment matrix associated with a sequence $y$, and let $M_{k}(G y)$ denote the localizing matrix that, by abuse of notation, can be written as
\begin{align}
    M_{k}(G y) &= L_{y}\left( \bs{b}_{k} \bs{b}_{k}^{\top} \otimes G\right),
\end{align}
$|\alpha|,|\beta| \leq k$, where $\otimes$ stands for the Kronecker product.

As in the scalar case, associated to a parent non-convex program, we consider the following $k$ truncated linear problem
\begin{align}
    f^{(k)}=\min_{y} & \quad L_{y}(f)\\
    \text{s.t.} & \quad  y_{0}=1 \\
                & \quad M_{k}(y) \succeq 0 \\
                & \quad M_{k-d}(G y) \succeq 0,
\end{align}
where $M_{k}(y)$ and $M_{k-d}(G y)$ are the $k$-th order truncated moment and localizing matrices, respectively, associated with the sequence $y$ and the matrix $G$. We say that this program is a relaxation of the PMI Prob. \eqref{problem:pmi}, with $f^{(k)}\leq f^*$ for all $k$. Using results from \cite{Hol04sumof,6f374c9990f14911a9cd629936a86c44}, Henrion and Lassere are able to prove that as $k \to \infty$, $f^{(k)}$ approaches $f^*$ generalizing the Lassere/SOS convergent relaxations to PMIs.


Since \cite{Henrion2006}, several improvements, based on different Positivstellens{\"a}tze, have been put forward, as summarized in Table \ref{table:summary-pop}. These state-of-the-art methods greatly improve the practicality of the methods of \cite{Henrion2006} and provide strong theoretical convergence guarantees. However, they do increase the complexity of the implementation considerably since they do not scale well given their super-exponential or even unbounded runtime originating, for example, on the non-necessarily bounded level of the relaxation hierarchy. Trying to find a computationally less resource intensive alternative, we formulate the notion of \emph{tame representations} of PMIs which do involve (weaker) theoretical convergence guarantees, like the Henrion-Lassere family of methods, for polynomial time methods such as the Newton method, gradient flows, and subgradients.

\subsection{PMIs as Non-linear SDPs}

PMIs should also be seen as non-linear SDPs \cite{yamashita2015survey}. There, first-order optimality conditions are well studied \cite{jarre2012elementary,andreani2020optimality,andreani2021optimality}. 
Although the direct application of the first-order optimality conditions \cite{fiala2013penlab} provides weaker \emph{local} convergence guarantees \cite{diehl2006loss,freund2007nonlinear,sun2008rate} than the global convergence guarantees of the previous section, 
they have a certain appeal to practitioners, who often prefer to rely on methods sacrificing   guarantees of convergence to global optima for the sake of time efficiency. 
Indeed, the present-best numerical methods for PMIs \cite{kovcvara2003pennon,kocvara2005penbmi,henrion2005solving} come only with local convergence guarantees. 



\begin{table*}[t]
\centering
\setlength{\tabcolsep}{6pt}
\begin{tabular}{l|ccc} \hline \hline
\multirow{2}{*}{Positivstellenatz} & \multirow{2}{*}{POP} & \multirow{2}{*}{\begin{tabular}[c]{@{}l@{}}PMIs\end{tabular}} & \multirow{2}{*}{\begin{tabular}[c]{@{}l@{}}Sparse PMIs\end{tabular}} \\
& & & \\ \hline
Hilbert-Artin & Artin   \cite{artin1927zerlegung}   & Du \cite{D2018ANO}  & Zheng \& Fantuzzi \cite[Thm. 2.2]{Zheng_2021} \\
Reznick       & Reznick  \cite{Reznick1995}  & Dinh \emph{et. al.} \cite{Dinh2021}  & Zheng \& Fantuzzi \cite[Thm. 2.3]{Zheng_2021}   \\
Putinar  & Putinar \cite{Putinar1993}  & Scherer \& Hol \cite{Scherer2005}  & Zheng \& Fantuzzi \cite[Thm. 2.4]{Zheng_2021} \\
  & Lasserre \cite{lasserre2001global}  & Henrion \& Lasserre \cite{Henrion2006}  &  \\
Putinar-Vasilescu  & Putinar \& Vasilescu  \cite{Putinar1999}   & Dinh \emph{et. al.}\cite{Dinh2021}   & Zheng \& Fantuzzi \cite[Thm. 2.5]{Zheng_2021}  \\ \hline
\end{tabular}
\caption{Summary of Positivstellens\"atze for polynomials and polynomial optimization problems (POP), polynomial matrix inequalities (PMIs), and PMIs with structural sparsity, expanded from Zheng and Fantuzzi \cite{Zheng_2021}.}
\label{table:summary-pop}
\end{table*}


\subsection{Representations in Optimization Theory}

It should be clear by now that PMIs form a generalization of LMIs, where the entries of the matrix $P$, as in Eq. \eqref{eq:LMIex} are linear functions of the indeterminates $x \in \mathbb{R}^n$. In the literature of optimization theory, an interesting question is under which conditions a subset of $\mathbb{R}^n$ can be represented as an LMI. Helton and Vinnikov considered this precice problem in \cite{Helton2003}. 

\begin{theorem}[LMI representations \cite{Helton2003}]
A subset $\mathcal{C} \in \mathbb{R}^{m}$, with 0 in the interior of $\mathcal{C}$, admits an LMI representation, if it is a convex algebraic interior and the minimal defining polynomial $p$ for $\mathcal{C}$ satisfies the so-called RZ condition (with $p(0)>0$), wherein 
 a real zero (RZ) polynomial $p(x)\in \mathbb{R}[x]$ for $x\in\mathbb{R}^m$ satisfies $p(\mu x)=0 \Rightarrow \mu \in \mathbb{R}$.
\end{theorem}



 
See also \cite{BRANDEN20111202}. 
Furthermore, in \cite{Helton2003} it is also shown that, when $m=2$, if $p$ is a RZ polynomial of degree $d$ and $p(0)>0$, then $\mathcal{C}_{p}$ has a monic LMI representation with $d \times d$ matrices.
In an important paper, Fawzi and El Din \cite{fawzi2018lower} bounded the size of the LMI representation of a convex set in dimension $d$
from below by $\Omega(\sqrt{\log d})$. 

The representation of a certain subset of $\mathbb{R}^n$ as the feasible set of a convex optimization problem or the representation of the feasible set of an optimization problem as a subset of $\mathbb{R}^n$ with some extra properties is a powerful concept. Nie, in \cite{Nie2009}, considered the necessary conditions such that a PMI can be represented as an SDP, essentially, finding a map that takes a non-convex problem to a convex one.

Related concepts of ``MPC equivalence'' \cite{ZANON2022110287}, ``reformulation'' \cite{adams2004comparisons,liberti2009reformulations},
and reduction \cite{ryoo1996branch}
have been studied in control theory, discrete optimization, and global optimization, respectively. 

Similarly to Helton and Vinnikov, one can ask when a subset of $\mathbb{R}^n$ can be represented as a PMI, a problem that, to the best of our knowledge, is open. However, we find strikingly more powerful applications in asking when a PMI can be represented as a subset $\mathcal{C}$ of $\mathbb{R}^n$ with some special ``tame" property (in the sense of definable functions in o-minimal structures, cf. Definition \ref{def:tamestructure}, p. \pageref{def:definable}), i.e., asking what are the conditions that a PMI is represented as a subset $\mathcal{C} \subset \mathbb{R}^n$ defined over an o-minimal structure.

\section{Tame representations of PMIs}\label{sec:tamereps}

In this section, we define the notion of a \emph{tame representation of a PMI} and give a few examples of such representations. We take the view of Lasserre, who presented a broad generalization of a Lagrangian in \cite{Lasserre2013ALR}, and introduce the following definition:
\begin{definition}[Generalized Lagrangian]
 A \emph{Generalized Lagrangian of a PMI} \eqref{problem:pmi} is a functional 
 \begin{equation}\label{eq:Lagrangian}
      \mathcal{L}_\lambda(x, y) := b(y) + \lambda \mathds{1}_{P(x,y) \not \succeq 0},
 \end{equation}
 where the scalar $\lambda > 0$ is known as the Lagrange multiplier and
 $\mathds{1}_{P(x,y) \not \succeq 0}$ denotes the indicator function defined as zero if $P(x,y) \not \succeq 0$ and one otherwise. 
\end{definition}

\begin{example}
\label{ex1}
Let us consider the planar PMI of \cite[Example II-E]{Henrion2006}, which is beautifully illustrated in \cite[Figure 1]{henrion2011inner}. In this case, the cost function is defined by $b(y) := y$ and the polynomial matrix $P$ is defined by
$$P(x,y) = \begin{bmatrix}  1- 16x y & x \\ x & 1 - x^2 - y^2 \end{bmatrix}.$$
The Generalized Lagrangian is then 
$y + \lambda \mathds{1}_{P(x,y) \not \succeq 0}$. The global minimum is attained at $(x,y)=(0,-1)$ and is $b(y)=-1$.
\end{example}

Note that the Generalized Lagrangian, as defined above, is a non-smooth function due to the indicator function $\mathds{1}_{P(x,y)}$.
At the same time, the Generalized Lagrangian is semialgebraic, in the sense that the indicator function of
a semialgebraic set is semialgebraic \cite[p. 26]{bochnak2013real}, but this is not particularly useful in terms of 
numerical optimization techniques.

Much of the traditional numerical optimization techniques consider a Lagrangian relaxation with a fixed, finite $\lambda$ and some smoothing of the indicator function.
Usual guarantees on the exactness of the Lagrangian do not extend to non-convex problems \eqref{problem:sdmi};
indeed, even if the polynomials $Q$ were convex (e.g., sums of squares), then the SDP need not satisfy strong
duality, which hinders the exactness of the Lagrangian.

Instead, one could consider the ``saddle-point problem'' \cite{benzi2005numerical,nie2021saddle}
\begin{align}
    \inf_{y}\sup_{\lambda}  \mathcal{L}_\lambda,
\end{align}
or, to simplify theoretical analyses, consider $\lambda = \infty$. By utilizing $\mathcal{L}_\infty$ we achieve a reformulation of Prob. \eqref{problem:pmi} as an unconstrained optimization problem
that is equivalent to the constrained PMI \eqref{problem:pmi}.

To talk about a tame representation PMI \eqref{problem:pmi}, with a generalized Lagrangian as defined in \ref{eq:Lagrangian} we first need to define what we mean by a representation. 
Informally, this is a reformulation of the problem.
Formally, we define

\begin{definition}[Representation]
\label{def:repre}
A map $f\colon \mathbb{R}^n \to \mathbb{R}$ is said to be represented by the map $g: \mathbb{R}^m \to \mathbb{R}$ if there exists a map $h: \mathbb{R}^n \to \mathbb{R}^m$ such that $f = g\circ h$ for all values $\mathbb{R}^n$.
\end{definition}

\begin{example}[Offset]
\label{ex:offset}
Let us continue with our planar Example \ref{ex1}, with
$y + \lambda \mathds{1}_{P(x,y) \not \succeq 0}$ which we can consider as being the map $f$ of Definition \ref{def:repre}. 
One may choose the map $g: \mathbb{R}^2 \to \mathbb{R}$ to be $1 + y + \lambda \mathds{1}_{P(x,y) \not \succeq 0}$ and $h: \mathbb{R}^2 \to \mathbb{R}^2$ to be the map $h(x, y) := (x, y - 1)$, because $y$ enters linearly in the objective of the PMI.
\end{example}

\begin{example}[Embedding]
\label{ex:embedding}
For the example \ref{ex1} we can also consider a function $g: \mathbb{R}^3 \to \mathbb{R}$, whose arguments $(x, y, z) \in  \mathbb{R}^3$ would still yield $y + \lambda \mathds{1}_{P(x,y) \not \succeq 0}$. The  function $h: \mathbb{R}^2 \to \mathbb{R}^3$ would introduce a $z = 0$. 
\end{example}

Clearly, the examples above are unusual, in the sense that we do not benefit from any improved continuity properties of the functions $g$.
Still, the functions $g$ represent function $f$, considering that there exists a function $h$ that makes it possible to obtain, for all $y \in \mathbb{R}^2$, $f(x)=g(h(y)) = g(y)$.


\begin{definition}[Minimizer Representation]
A function $f: \mathbb{R}^n \to \mathbb{R}$ is said to be minimizer-represented by $g: \mathbb{R}^m \to \mathbb{R}$ if there exists a function $h: \mathbb{R}^n \to \mathbb{R}^m$ such that $\inf f$ attains its value at $x \in \mathbb{R}^n$ (``global minimizer'') that is the same as $h$ applied to the global minimizer of $\inf g$, whenever the global minimizers of $f, g$ are unique. 
When the global minimizers of $f, g$ are not unique, we require that this holds true for one pair of the global minimizers. 
\end{definition}

\begin{example}[Irrelevant Offset]
Let us revisit our Example \ref{ex:offset}, where $g$ was $1 + y + \lambda \mathds{1}_{P(x,y) \not \succeq 0} $, i.e., adding the offset of 1. For minimizer representability, we could consider $h$ to be the identity, because the offset of 1 would not affect the global minimizer. 
\end{example}

In this paper, we will provide various representations of the Generalized Lagrangian of a PMI, wherein a representation of a PMI is called \emph{tame} if the graph of the representation of a Generalized Lagrangian 
 $\mathcal{L}_\lambda$
 of the PMI is definable in an o-minimal structure.
Precisely, we have the following definition:

\begin{definition}[Tame representation] 
Consider a PMI $X$ with generalized Lagrangian $\mathcal{L}_\lambda$. Assume that $\mathcal{L}_\lambda$ is represented by the function $\widetilde{\mathcal{L}}_\lambda$. Then, $X$ is said to be \emph{tame-representable} if the graph $\Gamma_{\widetilde{\mathcal{L}}}$ is definable in an o-minimal structure in the large $\lambda$ limit.
\end{definition}


In what follows, we present five distinct examples of tame representations for PMIs.



\subsection{Tame representation using the characteristic polynomial}\label{sec:charpol}

Recall that in Prob. \eqref{problem:pmi}, $P(x,y) \in \mathcal{S}^m$, i.e. it is a symmetric matrix whose entries are polynomials in the indeterminates $x\in \mathbb{R}^k$ and $y\in \mathbb{R}^l$. The condition $P(x,y)\succeq 0$, defines a semialgebraic set. To see this, we take the characteristic polynomial of the $m\times m$ symmetric matrix $P(x,y)$,
\begin{equation}\label{eq:charpol}
  \begin{aligned}
    p(t;x,y) &\coloneqq \det(t\mathds{1}_m - P(x,y)) \\
    & = t^m+\sum_{j=1}^m(-1)^jq_j(x,y)t^{m-j}.
\end{aligned}  
\end{equation}
for $m$ polynomials $q_j(x,y)$. The roots of the characteristic polynomial are the eigenvalues. Since the eigenvalues of a real symmetric positive semidefinite matrix are all real and non-negative, we have $q_j(x,y)\geq 0$ for all $j$ due to the Descartes' rule of signs, see \cite{Henrion2006}. We thus have:

\begin{proposition}
One representation of the Generalized Lagrangian of the PMI \eqref{problem:pmi} is
\begin{equation}\label{eq:LagrangianA}
      \widetilde{\mathcal{L}}_\lambda(x,y) := b(y) - \lambda\prod_{j=1}^m q_j(x,y) .
 \end{equation}
 in the large limit of $\lambda$.
\end{proposition}

 \begin{proof}
Note that even for a sufficiently large and finite $\lambda$, the minimum of $b(y) - \lambda\prod_{j=1}^m q_j(x,y)$ will have positive $q_j(x,y)$ for all $j$.
By extending Example \ref{ex:offset} to have the offset dependent on $x, y$, we can construct $h$ in Definition \ref{def:repre}.
 \end{proof}
 
 The obvious difficulty here is that in order to construct $h$ in Definition \ref{def:repre},
 we would have to know the global minimizer. 
 To address this difficulty, we can consider:
 
 \begin{proposition}
One minimizer-representation of the Generalized Lagrangian of the PMI \eqref{problem:pmi} is
\begin{equation}\label{eq:LagrangianB}
      \widetilde{\mathcal{L}}_\lambda(x,y) := b(y) - \lambda\prod_{j=1}^m q_j(x,y) .
 \end{equation}
 in the large limit of $\lambda$.
\end{proposition}





\begin{example} Let us study the preceding representation for the Example \ref{ex1}. The characteristic polynomial, \eqref{eq:charpol}, is 
\begin{equation*}
\begin{aligned}
p(t;x,y) &= t^2 - t (2-x^2 - 16 x y - y^2) \\
         & \quad + (16 x^3 y -2x^2-16xy+16xy^3-y^2+1)
\end{aligned}
\end{equation*}
We thus have the polynomials
\begin{equation*}
\begin{aligned}
q_1&=2-x^2 - 16 x y - y^2,\\
q_2&=16 x^3 y -2x^2-16xy+16xy^3-y^2+1,
\end{aligned}
\end{equation*}
and the minimizer-representation of the generalized Lagrangian is simply
\begin{equation*}
    \widetilde{\mathcal{L}}_{\lambda}=b(y)-\lambda q_1q_2.
\end{equation*}

\end{example}

The advantage of considering the minimizer-representation becomes apparent here: we only wish for the global minimizer to coincide, rather than for the value attained at the global minimizer to coincide. 


\begin{proposition}
The representation \eqref{eq:LagrangianA} of the generalized Lagrangian  is definable in a tame structure.
\end{proposition}
 
\begin{proof}
 The representation \eqref{eq:LagrangianA} is polynomial, and hence its graph is tame by definition.  
\end{proof}






Furthermore, semialgebraic sets are the smallest sets defined over an o-minimal structure, specifically $(\mathbb{R},+\times)$.
Therefore, the characteristic polynomial defines a tautological tame representation for the PMI of Prob. \eqref{problem:pmi}. 





The example given here is perhaps one of the simplest to construct, but it suffers from some severe drawbacks when it comes to practical use. Namely, for large $m$ the characteristic polynomial is a high-degree polynomial, which is inherently difficult (albeit not impossible, cf. \cite{Armentano2018}) to work with numerically. 

\subsection{Tame representation using factorisation}\label{sec:factorization}



Factorization is a very useful concept throughout mathematical optimization and machine learning. For example, solutions to problems that concern optimization over a large square matrix, such as the matrices involved in the attention mechanism in transformers \cite{Vaswani2017}, often utilize factorization into a pair of matrices of much
lower ranks. 

In addition, in the context of this article, the exploitation of the sparsity in Prob. \eqref{problem:sparse} can be seen as a factorization problem. 
As another example, one can replace the positive semidefinite constraint $X \in \mathcal{S}_+^m$
with the low-rank factorization $X\to vv^\top$, in an approach now known as the Burer-Monteiro 
technique \citep{Burer2003,burer2005local}. 

Given that a factorization is found, the analysis of optimization over such a factorization might be non-trivial, given that it converts a convex programming problem to a non-convex one. 
Consider that one of the globally-optimal solutions $X^*$ has some rank $r \geq 0$.
Instead of an $m \times m$ matrix $X$, we can consider its factorization
\begin{equation}\label{e:pmi-ours-factorized}
\begin{aligned}
X^* :=& \inf_{v} \quad Q(v v^\top),\qquad vv^\top=X.\\
&\text{s.t. }  \text{Tr}\left(v^\top v\right)=1,\\
\end{aligned}
\end{equation}
where $v \in \mathbb{R}^{r \times m}$. 

\begin{proposition}
A minimizer-representation of the generalized Lagrangian for the above problem is given by \cite{journee2008low}
\begin{align}\label{eq:LagrangianTr}
    \widetilde{\mathcal{L}}_\lambda(v)=Q(vv^\top)-\lambda\left(\mathrm{Tr}\left(v^\top v\right)-1\right),
\end{align}
if and only if
\begin{equation*}
    \nabla_X Q(vv^\top)-\lambda\eqqcolon S_v\succeq 0, 
\end{equation*}
for Lagrangian multipliers $\lambda$ satisfying
\begin{equation*}
    S_v v=0.
\end{equation*}
\end{proposition}

This is proven in \cite{journee2008low}.
Some of these analyses could be simplified considerably by considering:

\begin{proposition}
\label{prop:factorizationtame}
The graph of $Q(v v^\top)$ is semialgebraic and thus tame.
\end{proposition}

Of course, this is trivial due to the polynomial nature of the objective function. With Proposition \ref{prop:factorizationtame},
we thus know that there exists a descent guarantee along any subgradient trajectory \cite{Helton2003}, which could 
strengthen Theorem 4.1 in \cite{burer2005local}.
We also know that Newton's method converges fast on the factorization \cite{bolte2009tame},
which has been observed empirically \cite{liu2017hybrid}, but which has not been proven, yet.
Plausibly, more recent analyses \cite[e.g.]{boumal2016non,boumal2020deterministic}
could also benefit.

In practice, however, such a PMI representation is not particularly convenient since the optimal rank for the decomposition is not known a priori. Therefore, factorization provides a family of possible representations, parametrized by the rank $r$ of $v$. The actual representation $h$ required might be hard to estimate. 

\begin{remark}
Notice that this approach has rich links to numerical algebraic geometry
and its uses in linear semidefinite programming \cite{naldi2015exact,Din2016,naldi2018solving,henrion2021exact}. 
Notably, for $0\leq r \leq m-1$, \cite{naldi2015exact,Din2016,naldi2018solving,henrion2021exact} consider the real algebraic set
\begin{equation}
    \mathcal{D}_r\coloneqq \{ x\in \mathbb{R}^k\, :\, {\rm rank}(X)\leq r \}.
\end{equation}
This is an algebraic variety, known as the determinental variety. We expect that the representation through factorisation presented in this Section would mainly be appropriate for problems with rank deficient optimizers \cite{journee2008low}, i.e., where the optimum lies in $\mathcal{D}_r$. 
\end{remark}










\subsection{Tame representation using log-det barrier}\label{sec:logdet}

Perhaps the most natural tame representation utilizes the log-det barrier function $\phi$
\cite[Proposition 5.4.5]{nesterov1994interior}
\cite[Chapter 11]{Boyd2004}, which for a positive definite matrix $X$ is defined by
\begin{equation}
    \phi(X)\coloneqq\begin{cases}-\log \det X,\qquad &X\succ0, \\ +\infty,\qquad &\text{otherwise.}\end{cases}
\end{equation} 
This is known \cite[Remark 5.4.1]{nesterov1994interior} to be the best possible self-concordant barrier 
function for the cone of positive semidefinite matrices. 



As above, the idea in using the log-det barrier is to reformulate the PMI Prob. \eqref{problem:sdmi}
as an unconstrained problem.
We reformulate the Generalized Lagrangian $\widetilde{\mathcal{L}}_\lambda$ as 
\begin{equation}
    \widetilde{\mathcal{L}}_\lambda = b(X)+ \lambda \phi(X).
\end{equation}
Essentially, the log-det barrier approximates the non-smooth indicator function of the Generalized Lagrangian \eqref{eq:Lagrangian} with a smooth function. The problem to solve becomes
\begin{align}
    \inf_{X}\,\,\widetilde{\mathcal{L}}_\infty.
\end{align}

At the same time:

\begin{theorem}
The map $ b(X)+ \lambda \phi(X)$ is definable in a tame structure.
\end{theorem}
\begin{proof}The log function is the inverse of the exp and a definable map due to Example \ref{ex: exp} and 
Example \ref{def:tamestructure}. Hence, the graph of $\phi$ is tame as it is the composition of two definable maps. To conclude the proof, recall that $b(X)$ is a polynomial hence a definable map.
\end{proof}

With this result, we conclude that the log-det barrier function provides a tame representation of the PMI (Prob. \ref{problem:sdmi}).

Notice that leading first-order solvers for PMI, such as PENBMI \cite{kovcvara2003pennon,kocvara2005penbmi,henrion2005solving},
come with \emph{ad hoc} proofs of local convergence, without using the results from tame geometry \cite{Kurdyka2000}.
At the same time, results from tame geometry \cite{bolte2009tame} could be used to show that Newton's 
method converges fast, for instance. 







\subsection{Tame representations using alternative barrier functions}\label{sec:barrier2}

While the log-det barrier function is the natural choice for numerical optimization methods \cite{kocvara2005penbmi,henrion2005solving},
a wealth of other options exist \cite[Chapter 5]{nesterov1994interior} \cite[Section 3.1]{fiala2013penlab}. 
Let us consider rational powers of the determinant, instead of the logarithm of the determinant above: 
\begin{equation}
    \phi(X)\coloneqq\begin{cases}- \det^{r} X,\qquad &X\succ0, \\ +\infty,\qquad &\text{otherwise.}\end{cases}
\end{equation} 
for some rational $0 < r < 1$ following Example 3.1 of \cite{nesterov1994interior}. 
This is \cite[p. 239]{nesterov1994interior} positive-definite representable, which should allow for analyses similar to the sections above, e.g., Section \ref{sec:charpol}.


\subsection{Tame representation using a bound}\label{sec:boundbased}

In PMI (Prob. \ref{problem:pmi}), let us denote the global optimum of the objective function $b^*$.
If we knew the optimum $b^*$, we could constrain the feasible set of the optimization problem 
such that all feasible points were the global optima:
$\big\{P(x,y) \succeq 0 \big\} \cap \big\{ b(y) \le b^*\big\}$.

We do not know $b^*$ a priori, but we can work with an estimate $\hat b$. 
Then, repeatedly deciding the emptiness of the intersection: 
\begin{align}
\label{eq:intersection}
\big\{P(x,y) \succeq 0 \big\} \cap \big\{ b(y) \le \hat b \big\}
\end{align}
is clearly equivalent to solving the original problem \eqref{problem:pmi}. 
Trivially, one may consider, e.g., bisection on $\hat b$:
whenever the intersection \eqref{eq:intersection} is empty, multiply $\hat b$ by a number between 1 and 2.
Whenever the intersection \eqref{eq:intersection} is non-empty, halve $\hat b$.

Thus: 

 \begin{proposition}
There exists $\hat{b}$ such that one minimizer-representation of the Generalized Lagrangian of the PMI \eqref{problem:pmi} is
\begin{equation}
\begin{aligned}
\label{eq:boundbased}
\widetilde{\mathcal{L}}_{\hat b}(x,y) := \min_{x,y} & \quad 1\\
    {\rm s.t. } & \quad P(x,y) \succeq 0 \\
                & \quad b(y) \le \hat b.
\end{aligned}
\end{equation}
\end{proposition}

Such a representation suffers from a number of issues:
\begin{itemize}
    \item the fact that $\hat b$ is hard to find, in general;
    \item testing feasibility of semidefinite programming is hard, in general;
    \item while the intersection of two semialgebraic sets is semialgebraic,
    and $P(x,y) \succeq 0$ is semialgebraic, we may need one of the representations of Sections \ref{sec:charpol}--\ref{sec:barrier2}
 in a numerical routine. 
\end{itemize}
Still, the convexifications of Section \ref{sec:convexifications} can be seen as 
using this minimizer representation by asking for the 
non-negative polynomial $b(y) - b^*$ to belong to the truncated quadratic module.
Thus, this representation may be one of the most often used, in practice.

\section{Conclusions}\label{sec:discussion}

We have shown how to reformulate PMIs \eqref{problem:pmi} within tame geometry, to make it possible to utilize first-order algorithms with recently developed convergence guarantees \cite{davis2020stochastic,bolte2021conservative}.  
The interest for such reformulations lies precisely in the fact that if a representation of Prob. \eqref{problem:pmi} is Whitney stratifiable, then first-order methods might be preferable compared to the convexifications of Sec. \ref{sec:PMIs}.

One obvious question that we can ask is what representation is best suited to a particular application. 
It seems plausible that the representations of Secs. \ref{sec:charpol} and \ref{sec:logdet} are suitable only for full-rank optimizers, due to the nature of the Generalized Lagrangian. Similarly, the representation of Sec. \ref{sec:factorization} would be most useful for low-rank minimizers.
As we can tell from Corollary 3.1 in \cite{amelunxen2015intrinsic}, the low-rank minimizers may actually 
be quite rare, in general, although they are very desirable (and more common) in particular applications.
See also \cite{fawzi2018lower}.
If we do not have reasons to believe that the low-rank minimizers are available in a particular instance, perhaps Secs. \ref{sec:charpol} and \ref{sec:logdet} are a great starting point.

An independent consideration may arise when one considers an intersection with further linear matrix inequalities. 
The work of Waldspurger and Waters \cite{Waldspurger2020} showed that if the size of the factors is less than the square root of the number of linear constraints, poor behavior may arise in the representation of Sec. \ref{sec:factorization}, even if we deal with linear matrix inequalities. 
This may suggest that the representation of Sec. \ref{sec:factorization} is particularly suitable for problems with few linear constraints in such extensions of polynomial matrix inequalities. 

A related question asks how much we can relax the large $\lambda$ limit constraints in the representations. From some numerical experiments, it seems hard to attain the true global minimizers using any of the Generalized Lagrangians we have presented here for a finite value of $\lambda$, but for any given $\epsilon$ error, there is a finite value of $\lambda$ that suffices. 
Bounding the value of $\lambda$ as a function of $\epsilon$ such that we can guarantee that it suffices seems rather hard analytically. 
A more in-depth experimental study, or perhaps a ``how-to'' procedure, would be a good venture for future research. 

Many natural questions arise in connection with the notion of a tame representation of sets, more generally. 
A question, probably of mathematical flavor, is whether the map $h$ is unique or not.
Another natural question is whether every map $f$ is thus representable by some $h$, or what conditions ensure that $f$ is representable.


\begin{acknowledgments}
The authors thank Didier Henrion for useful comments, insights, and suggestions.

J.A. is supported by the Government of Ireland Postgraduate Scholarship Programme GOIPG/2020/910 of the Irish Research Council. G.K. and J.M. acknowledge support of the OP RDE funded
project CZ.02.1.01/0.0/0.0/16\_019/0000765 “Research Center for Informatics”. 
J.M. acknowledges support of the
Czech Science Foundation (22-15524S).
\end{acknowledgments}

\bibliography{pmi}

\end{document}